\documentclass[12pt,centertags,oneside]{amsart}
\usepackage{amsmath,amstext,amsthm,amscd,typearea,hyperref}
\usepackage{amssymb}
\usepackage{a4wide}
\usepackage[mathscr]{eucal}
\usepackage{mathrsfs}
\usepackage{typearea}
\usepackage{charter}
\usepackage{pdfsync}
\usepackage[a4paper,width=16.2cm,top=3cm,bottom=3cm]{geometry}

\numberwithin{equation}{section}

\usepackage{xcolor}

\newtheorem{theorem}{Theorem}[section]
\newtheorem{definition}[theorem]{Definition}
\newtheorem{proposition}[theorem]{Proposition}
\newtheorem{corollary}[theorem]{Corollary}
\newtheorem{lemma}[theorem]{Lemma}

\newtheorem{conjecture}[theorem]{Conjecture}

\newcommand{\cali}[1]{\mathscr{#1}}

\newcommand{\U}{{\rm U}}

\newcommand{\volume}{{\rm vol}}

\newcommand{\supp}{{\rm supp}}
\newcommand{\diff}{{\rm d}}

\newcommand{\dist}{\mathop{\mathrm{dist}}\nolimits}

\newcommand{\ddc}{{\rm dd}^c}

\newcommand{\Tub}{{\rm Tub}}
\newcommand{\omegaFS}{ \omega_{\mathrm{FS}}}

\usepackage{mathtools}
\DeclarePairedDelimiter\ceil{\lceil}{\rceil}
\DeclarePairedDelimiter\floor{\lfloor}{\rfloor}

\newcommand{\PSH}{{\rm PSH}}

\newcommand{\PC}{{\rm PC}}

\newcommand{\Cc}{\cali{C}}
\newcommand{\Dc}{\cali{D}}
\newcommand{\Ec}{\cali{E}}
\newcommand{\Fc}{\cali{F}}

\newcommand{\FS}{{\rm FS}}
\newcommand{\B}{\mathbb{B}}
\newcommand{\D}{\mathbb{D}}
\newcommand{\C}{\mathbb{C}}

\newcommand{\K}{\mathbb{K}}
\newcommand{\N}{\mathbb{N}}
\newcommand{\Z}{\mathbb{Z}}
\newcommand{\R}{\mathbb{R}}

\renewcommand\P{\mathbb{P}}

\newcommand{\W}{\mathbb{W}}
\newcommand{\I}{\mathbb{I}}
\newcommand{\bbS}{\mathbb{S}}

\newcommand{\rmC}{{\rm C}}
\newcommand{\diam}{{\rm diam}}

\newcommand{\fB}{\mathfrak{B}}
\newcommand{\fW}{\mathfrak{W}}

\title[Exponential equidistribution of periodic points]{Exponential equidistribution of periodic points for endomorphisms of $\P^k$}

\begin{author}[H. de Th\'elin]{Henry de Th\'elin}
\address{Universit\'e Paris 13, Sorbonne Paris Nord, LAGA, CNRS (UMR 7539),
F-93430, Villetaneuse, France. }
\email{ dethelin@math.univ-paris13.fr}
\end{author}

\begin{author}[T.-C. Dinh]{Tien-Cuong Dinh}
\address{National University of Singapore, Lower Kent Ridge Road 10,
Singapore 119076, Singapore}
\email{matdtc$@$nus.edu.sg }
\end{author}

\begin{author}{Lucas Kaufmann}
\address{Institut Denis Poisson, CNRS, Universit\' e d'Orl\' eans, Rue de Chartres, B.P. 6759, 45067, Orl\' eans cedex 2, FRANCE}
\email{lucas.kaufmann@univ-orleans.fr}
\end{author}


\allowdisplaybreaks
\emergencystretch=\maxdimen
\begin{document}

\begin{abstract}
Let $f$ be a holomorphic endomorphism of $\P^k$ of algebraic degree $d\geq 2$.  We show that the periodic points of $f$ of period $n$ equidistribute towards the equilibrium measure of $f$ exponentially fast as $n$  tends to infinity. This quantifies a theorem of Lyubich for $k=1$ and of Briend-Duval for $k\geq 2$.  A byproduct of our proof is the existence of a large number of periodic cycles in the small Julia set with large multipliers.
\end{abstract}

\maketitle

\section{Introduction} \label{s:intro}

Let $\P^k$ be the $k$-dimensional complex projective space and $f: \P^k \to \P^k$ be a holomorphic endomorphism of algebraic degree $d \geq 2$,  see Section \ref{s:prelim} for the basic definitions appearing in this introduction.   It follows from the seminal works of Brolin \cite{brolin}, Freire-Lopes-Ma\~n\'e \cite{freire-lopes-mane} and Lyubich \cite{lyubich2, lyubich}  when $k=1$ and Fornaess-Sibony \cite{fornaess-sibony:93} and Hubbard-Papadopol \cite{hubbard-papadopol} in general that $f$ admits a canonical invariant measure $\mu$, called the \emph{equilibrium measure of $f$}.  This is a probability measure on $\P^k$ characterized by various dynamical properties.  In particular, it is exponentially mixing and it is the unique invariant probability measure of maximal entropy,  see \cite{dinh-sibony:cime} for an overview.   We highlight two important properties satisfied by $\mu$,  namely two equidistribution theorems.

The first equidistribution theorem concerns iterated pre-images of non-exceptional points,   see \cite{brolin,freire-lopes-mane,lyubich,briend-duval:IHES,dinh-sibony:cime}.  More precisely, there exists an algebraic exceptional set $\Ec \subset \P^k$ such that
\begin{equation} \label{eq:equidist-preimage}
  \lim_{n \to \infty} \frac{1}{d^{kn}} \sum_{x \in f^{-n}(a)} \delta_x = \mu \quad \text{for every} \quad  a \in \P^k \setminus \Ec,
\end{equation}
where the convergence is in the weak sense.    Sibony and the second author showed that the above convergence can be quantified and  is exponentially fast in the following sense.  Let  $1 < \lambda < d^{1/2}$.  There exists a constant  $C_{a,\lambda} >0$ and a proper algebraic subset $\Ec_\lambda$ of $\P^k$ such that if $a \in \P^k \setminus \Ec_\lambda$,  then  for every test function $\phi$  on $\P^k$ of class $\Cc^1$ we have
\begin{equation} \label{eq:equidist-preimage-speed}
\bigg| \Big \langle  \frac{1}{d^{kn}} \sum_{x \in f^{-n}(a)} \delta_x,  \phi \Big \rangle  - \int_{\P^k} \phi \, \diff \mu \bigg| \leq  C_{a,\lambda} \|\phi \|_{\Cc^1} \frac{1}{\lambda^n} \, ,
\end{equation}
where $\delta_x$ is the Dirac mass at $x$.
Moreover,  the constant $C_{a,\lambda}$  is proportional to  $ \big(1 + \log^+ (1 / \dist(a,  \Ec_\lambda)\big)^{1/2}$, where $\log^+:=\max(\log,0)$.  See \cite{dinh-sibony:equid-speed},  or Theorem \ref{t:eq-pt} below for another version of this result.

A second equidistribution theorem satisfied by $\mu$ concerns the periodic points of $f$. This is  a fundamental result of Briend-Duval \cite{briend-duval:acta},  obtained previously by Lyubich when $k=1$ \cite{lyubich}.  
For each $n \geq 1$, let $P_n := \{ x \in \P^k : f^n(x) = x \}$ be the set of periodic points of period $n$ of $f$. Then,   
\begin{equation} \label{eq:equidist-periodic-points}
 \lim_{n \to \infty} \frac{1}{d^{kn}} \sum_{a \in P_n} \delta_a = \mu,
\end{equation}
where again the convergence is in the weak sense.   Differently from the case of pre-images,  the estimation on the speed of convergence in \eqref{eq:equidist-periodic-points} is a more challenging problem.  When $k=1$,  Favre--Rivera-Letelier and Okuyama obtained an exponential speed of convergence when $f$ is defined over a number field,  see \cite{favre-rivera-letelier,okuyama}.  The rate of convergence in these papers is likely optimal. In an earlier version of this paper, we obtained the exponential speed for all polynomial maps in $\C$.
In a very recent preprint,  Gauthier-Vigny  \cite{gauthier-vigny:periodic} generalize the techniques of Favre-Rivera-Letelier to cover the case of any rational function,  not necessarily one with algebraic coefficients.  In higher dimensions,  
Yap \cite{yap} obtained an exponential speed of convergence for endomorphisms of $\P^2$ defined over a number field and in \cite{DY25}, Yap and the second author extended this result to higher dimensions, still for maps defined over a number field.

In this work, which replaces the previous arXiv version for polynomials, we solve this question in any dimension.   Moreover, our proof shows that most of the points of $P_n$ are repelling and belong to the small Julia set.

\begin{theorem} \label{t:main}
Let $f$ be a holomorphic endomorphism of $\P^k$ of algebraic degree $d\geq 2$,  $\mu$ be its equilibrium measure and $J_k:=\supp(\mu)$ be its small Julia set.   Then,  as $n$ tends to infinity, the repelling periodic points of period $n$ of $f$ on $J_k$ are equidistributed with respect to $\mu$ at an exponential rate.

More precisely,  for every constants $0<\gamma<1$ and $0<\alpha \leq 1$,  there exist a constant  $0<\xi<1$ independent of $\alpha$ and another constant $A_\alpha>0$ such that the following holds.
 
Let $P_n$ be the set of periodic points of period $n$ of $f$. Let $P_{n,\gamma}$ be the set of points $a\in P_n\cap J_k$ such that $\|Df^n(a)^{-1}\|\leq d^{-\frac{1-\gamma}{2}n}$ and $Q_n$ be any set such that $P_{n,\gamma}\subset Q_n\subset P_n$,  counting  multiplicities or not.  Then, we have
$$\Big| \Big\langle\frac{1}{d^{kn}}\sum_{a\in Q_n} \delta_a -\mu, \phi \Big\rangle \Big| \leq A_\alpha \xi^{\alpha n} \|\phi\|_{\Cc^\alpha},$$
for any $\Cc^\alpha$ test function $\phi$ on $\P^k$, where $\delta_a$ denotes the Dirac mass at $a$.
\end{theorem}

 The constant $\xi$ above can be made explicit in terms of $d$ and the H\"older exponent of the Green function of $f$,  but this is not optimal as the above mentioned results for $k=1$ show.  The condition $\|Df^n(a)^{-1}\|\leq d^{-\frac{1-\gamma}{2}n}$ implies that the multipliers of the $n$-cycle containing $a$ are bounded from below by $d^{\frac{1-\gamma}{2}n}$.  This control is likely optimal.

When $k=1$ every repelling periodic point is in the Julia set.  Moreover, the Fatou-Shishikura inequality asserts that there are at most $2d-2$ nonrepelling cycles.  However,  when $k\geq 2$,  Fornaess-Sibony showed that some endomorphisms may have infinitely many repelling periodic points outside $J_k$,  see \cite{fornaess-sibony:examples}.  Our proof allows us to estimate the number of such points.  

\begin{corollary} \label{c:main}
Let  $A_n$  denote the number of non-repelling periodic points of order $n$ of $f$ and $B_n$  the number of periodic points of order $n$ outside the small Julia set $J_k$, counting  multiplicities.  Then, there exists a constant $0<\xi<1$  such that
$A_n = O(\xi^n d^{kn})$ and $B_n = O(\xi^n d^{kn})$ as $n$ tends to infinity. 
\end{corollary}

We note that the periodic points with minimal period  $n$ also equidistribute towards $\mu$ with exponential speed.  This is because  the number of points of $P_n$ having a period $p$ diving $n$ is $d^{kp} + O(d^{(k-1)p})$  and $\sum_{p|n} \big( d^{kp} +O(d^{(k-1)p}) \big) \lesssim d^{kn/2}$ which is exponentially small compared to $d^{kn}$.

\smallskip

 We now describe the overall structure of the proof.  It is based on the construction of good inverse branches with a  control over their geometry.  Given $f: \P^k \to \P^k$ denote by $\rmC_f$ its critical set and by $\PC_m:=f(\rmC_f)\cup f^2(\rmC_f)\cup \cdots \cup f^m(\rmC_f)$ its post-critical set of order $m$. 

 We start by fixing a good atlas $\Omega_1, \ldots,  \Omega_M$ of $\P^k$ and, for a small parameter $r>0$,  a suitable covering of $\P^k$ by cells  $$\W_{r,\eta}^{k,j}, \quad  \eta \in \Z^{2k}, \,\,   j= 0, \ldots,M$$
 that are biholomorphic to cubes in $\C^k \simeq \R^{2k}$,  see Section \ref{s:branch}.   The side length $r$ of these cubes depend on $n$ and shrink exponentially to zero when $n$ tends to infinity.  
 
 We first show that we can discard cells touching a small neighborhood of $\PC_{\floor{\sigma n}}$ for some small constant $\sigma>0$.  For this,  we show that the mass of $\mu$ over neighborhoods of analytic subsets of $\P^k$ is small (Proposition \ref{p:neigh-mes}).   Since the degree of $\PC_{\floor{\sigma n}}$ grows exponentially with $n$, one needs a fine estimate,  which is not a simple consequence of the fact that $\mu$ is moderate. 
 
    Once the cells close to $\PC_{\floor{\sigma n}}$ are removed,  one can produce many inverse branches with good control on their diameters (see Proposition \ref{p:inverse}).  Our goal is then to show that on the cells  intersecting $J_k$ one can find many inverse branches of $f^n$ mapping $\W_{r,\eta}^{k,j}$ to a smaller cell $(1-r)\W_{r,\eta}^{k,j}$.  By Kobayashi hyperbolicity, this  yields a fixed point of $f^n$ inside this cell together with a control on the derivative,  producing a repelling periodic point of $f$ on $J_k$ with large derivative.  In order to show that cells are mapped to smaller ones,   we  must use the quantitative equidistribution \eqref{eq:equidist-preimage-speed}.  This  gives a quantifiable  way of controlling the mixing between cells and determines how many of them must been thrown away,  which is crucial if we are searching for the speed of convergence.     This is where the particularly simple geometry of our covering is useful.

\subsection*{Acknowledgments}
This project has received funding from the National University of Singapore and MOE of Singapore through the grants A-8002488-00-00.  The second author would like to thank Paris 13 University and the University of Orléans for their hospitality during his visits to these institutions. Part of this work was conducted during the third author's visits to the Department of Mathematics at the National University of Singapore and the Institute for Mathematical Sciences, Singapore. He gratefully acknowledges their warm hospitality and financial support.  He also received funding from the “Loi de programmation de la recherche” through the Université d’Orléans.

\section{Notations  and preliminaries} \label{s:prelim}

We introduce in this section some notations and basic results needed in the sequel.  We refer to \cite{dinh-sibony:cime} for more details.

Let $\P^k$ be the $k$-dimensional complex projective space.  By definition,  it is the quotient of $\C^{k+1} \setminus \{0\}$ by the equivalence relation $z \sim \lambda z$ for $\lambda \in \C^*$.  If $z=(z_0, \ldots,z_k) \in \C^{k+1}$, we will denote by $[z] = [z_0: \cdots: z_k]$ its equivalence class and call them \textit{homogeneous coordinates}.  For every $j=0,\ldots,k$ the open sets $U_j = \{z_j \neq 0\} = \{z_j = 1\} \subset \P^k$ are biholomorphic to $\C^k$.  We call them\textit{ standard affine charts}.

The Fubini-Study metric on $\P^k$ is,  up to a multiplicative constant, the unique $\U(k+1)$-invariant hermitian metric on $\P^k$. This is a K\"ahler metric whose associated hermitian form is the Fubini-Study form ---  a positive closed smooth  $(1,1)$-form that we denote by $\omegaFS$.  We normalize it so that $\int_{\P^k} \omegaFS^k = 1$.  In particular,  it follows that $\omegaFS^k$ is a smooth probability measure on $\P^k$.

If $\Omega$ is an open subset of $\P^k$,  a function $\varphi: \Omega \to \R \cup \{-\infty\}$, not identically $-\infty$ in any connected component of $\Omega$, is {\it plurisubharmonic (p.s.h.)} if it is upper semicontinuous and if the restriction of $\varphi$ to every holomorphic disc inside $\Omega$ is subharmonic or identically $-\infty$.  A function $u: \P^k \to \R \cup \{-\infty\}$ is {\it quasi-plurisubharmonic (q.p.s.h.)} if it is locally the difference of a p.s.h.  function and a smooth one.  If furthermore $ \omegaFS + \ddc u$ is a positive closed current, we say that $u$ is \textit{$\omegaFS$-plurisubharmonic }and we write $\omegaFS$-p.s.h. We denote by $ \PSH(\P^k,\omegaFS)$ the set of all such functions.

Let $f: \P^k \to \P^k$ be a holomorphic endomorphism.  In homogeneous coordinates it is given by $f=[P_0: \cdots: P_k]$ where the $P_j$ are homogeneous polynomials of the same degree $d$ without common zeros on $\C^{k+1} \setminus \{0\}$.  We call $d$ the \textit{algebraic degree} of $f$ and always assume that $d \geq 2$.  It can be shown using B\'ezout's theorem that the \textit{topological degree} of $f$, that is,  the number of pre-images of a given point $a \in \P^k$ counted with multiplicity, is exactly $d^k$.  The same theorem can be used to show that the number of periodic points of period $n$ counted with multiplicities is $(d^{(k+1)n}-1) / (d^n-1) = d^{kn} + O(d^{(k-1)n})$.

The \textit{Green current} of $f$ is the positive closed $(1,1)$-current defined by  $$T:= \lim_{n \to \infty} \frac{1}{d^n} (f^n)^*\omegaFS.$$
It can be shown that $T$ has H\"older continuous potential in the sense that $T = \omegaFS + \ddc g$ for some  $\omegaFS$-p.s.h.  function $g$ that is H\"older continuous on $\P^k$.  In particular, Bedford-Taylor's theory applies and one can define the associated Monge-Amp\`ere mesasure
$$\mu:= T^k = ( \omegaFS + \ddc g)^k.$$
This is the so-called \textit{equilibrium measure} of $f$.  The support of $\mu$,  denoted by $J_k$ is called {\it small Julia set} of $f$.  Clearly, it is contained in the {\it Julia set} $J_1:= \supp (T)$.  The complement $\P^k \setminus J_1$ is the {\it Fatou set},  i.e.,  the domain of normality of the family of iterates $(f^n)_{n \in \N}$.  The sets $J_1$ and $\P^k \setminus J_1$ will not play an important role in this work.

\section{Mass of Monge-Amp\`ere measures near analytic sets}

In this section we obtain an estimate of the mass of Monge-Amp\`ere measures on neighbourhoods of subvarieties of $\P^k$.  In the proof of our main theorem, we will apply this estimate to a high-order postcritical set of $f$,  see Section \ref{s:branch}.

Let $V \subset \P^k$ be a closed subset.  For $\varepsilon >0$ we will denote by
$$\Tub(V;\varepsilon):= \{x \in \P^k: \dist(x,V) < \varepsilon\}$$
the $\varepsilon$-tubular neighborhood of $V$, where $\dist$ denotes the distance induced by the Fubini-Study metric.

We say that a probability measure $\nu$ on $\P^k$ is a   \textit{Monge-Amp\`ere measure with H\"older continuous potential} if it is of the form $$\nu = (\omegaFS + \ddc u)^k,$$ where $u$ is $\omegaFS$-p.s.h. and H\"older continuous.  The equilibrium measure $\mu$ is an example of such a measure.  By the main result in \cite{DNS:JDG},  such measures are \textit{moderate},  that is,  there are constants $\beta > 0$ and $C>0$ depending only on $k$ and the H\"older exponent of $u$ such that
$$\int_{\P^k} e^{-\beta \varphi} d\nu \leq C \quad \text{for all } \varphi \in \PSH(\P^k,\omegaFS) \, \text{ such that } \max \varphi = 0.$$

We will also need to work in the local setting, that is, with local Monge-Amp\`ere measures of  the form $\nu = (\ddc v)^k$ where $v$ is a H\"older continuous p.s.h.\ function defined on an open subset $\Omega$ of $\C^k$. By \cite[Corollary 4.3]{kaufmann:skoda},  there is a constant $c(v) > 0$ depending only on $k$ and on the H\"older exponent of $v$ such that the following holds.  If $K \subset \Omega$ is a compact subset and $\varphi$ is a p.s.h.\ function on $\Omega$ whose Lelong numbers on $K$ are bounded from above by $L >0$,  then for every $\gamma < c(v) L^{-1}$ there exists $C_{\varphi,  \gamma} > 0$ such that
\begin{equation} \label{eq:skoda-sublevel}
\nu(\{ \varphi< - M\} \cap K) \leq C_{\varphi,  \gamma}  e^{-\gamma M}, \,   \text{for every } M > 0.
\end{equation}

Recall that the degree of  a proper subvariety $V \subset \P^k$ of dimension $\ell$ is the number of points  in the intersection $V \cap H$,  where $H$ is a generic projective subspace of $\P^k$ of dimension $k-\ell$.

The main result of this section is the following.

\begin{proposition} \label{p:neigh-mes}
Let $\nu$ be a  Monge-Amp\`ere probability measure with H\"older continuous potential  on $\P^k$.  There exists a constant $\kappa \geq 1$ such that, for any proper subvariety $V \subset \P^k$ of degree $\leq \delta$,  we have
$$\nu(\Tub(V;\delta^{-\kappa})) \leq \delta^{-1}.$$
The constant $\kappa$ depends only on $k$ and on the H\"older exponent of the potential of $\nu$.
\end{proposition}

The strategy of the proof is the following. At first, we observe that we can reduce to the case where $V$ is an algebraic hypersurface of degree $\delta$. The second step is to prove a local version of the desired estimate for hypersurfaces given by local holomorphic graphs. The third step, the most technical one, consists of considering a central projection from a point $a \in \P^k$. Away from a small set of ramifications, one can see $V$ as a disjoint union of graphs and use step two. In order to control the ramifications we use an induction step that will be explained below.

\smallskip

We start by considering  the  case of projective subspaces.

\begin{lemma} \label{l:H-neigh}
Let $\nu$ be a moderate probability measure on $\P^k$. There are constants $A>0$ and $\beta>0$ such that for every proper projective subspace $H \subset \P^k$ and $0<t<1$ one has $$\nu(\Tub(H;t))\leq A t^\beta.$$
\end{lemma}
\begin{proof}
After replacing $H$ by a hyperplane containing it, we can assume that $H$ has dimension $k-1$. Let $u_H$ be the unique $\omega_\FS$-p.s.h.\ function on $\P^k$ such that 
$$\ddc u_H = [H]-\omega_\FS \quad \text{ and } \quad \max u_H =0.$$
It is not difficult to deduce from the above equation that $u_H\leq \log \dist(\cdot,H)+c$ for some constant $c>0$ independent of $H$.

Since $\nu$ is  moderate, we have  that  $\langle \nu, e^{-\beta u_H}\rangle \leq C$ for some constants $\beta > 0$ and $C>0$.  It follows that  $\langle \nu, e^{-\beta \log\dist(\cdot, H) -\beta c}\rangle \leq C$.   Since $-\log\dist(\cdot, H) \geq -\log t$ on $\Tub(H;t)$, we deduce that $e^{-\beta \log t -\beta c} \nu(\Tub(H;t))  \leq C$. So, we have $\nu(\Tub(H;t)) \leq C e^{\beta c}  t^\beta$ and the lemma follows.
\end{proof}

We now treat the case of local holomorphic graphs. Let $A >0$ and $0< \alpha \leq 1$. In what follows, we say that a function $v$ defined on an open subset $\Omega$ of $\C^k$ is 
{\it $(A,\alpha)$-H\"older continuous} if
$$|v(x) - v(y)| \leq A \|x-y\|^\alpha \quad \text{for all } x,y \in \Omega.$$

\begin{lemma} \label{l:Gamma-neigh}
Let $v$ be an $(A,\alpha)$-H\"older continuous p.s.h. function on $10\D^k$. Let $h$ be a holomorphic function on $10 \D^{k-1}$ such that $|h(x)| \leq 1$ for all $x \in 10\D^{k-1}$. Denote by $$\Gamma_h = \{ (x,y) \in 10\D^{k-1} \times 10\D: y = h(x)\}$$ the graph of $h$ and for $0<t<1$ let
$$W_t:=\big\{ (x,y) \in 2\D^{k}, \quad |y- h(x)|<t \big\}.$$
 Then, there are constants $A'>0$ and $\alpha'>0$ depending only on $k, A, \alpha$ such that for every $0<t<1$ the mass of $(\ddc v)^k$ on $W_t$ is bounded by $A' t^{\alpha'}$. 
\end{lemma}

\begin{proof}
We first observe that we can reduce the lemma to the case $h=0$. In order to see that, consider the holomorphic map $\Phi(x,y) = (x,y-h(x))$ for  $(x,y) \in 10\D^{k}$. Then $\Phi$ is biholomorphic onto its image, it maps $\Gamma_h$ to $\{y=0\}$ and its inverse is $\Phi^{-1}(x,y) = (x,y+h(x))$.  The mass of $(\ddc v)^k$ over $W_t$ equals the mass of $(\ddc (v \circ \Phi^{-1}))^k$ over $\Phi(W_t)$. By Cauchy's formula, since $\|h\|_{\infty} \leq 1$  the derivatives of order 1 of $h$ over $5\D^{k-1}$ are bounded by a constant independent of $h$, so the Jacobians of $\Phi$ and $\Phi^{-1}$ are bounded by a constant independent of $h$.  In particular, if $v$ is $(A,\alpha)$-H\"older continuous then $v \circ \Phi^{-1}$ is $(CA,\alpha)$-H\"older continuous for some constant $C >0$ independent of $h$. We can therefore assume that $h=0$ as claimed.  Because of $\Phi$, we can  assume that $v$ is defined only on $10\D^{k-1}\times 9\D$. 

Assume from now on that $h=0$.  In this case $\Gamma_h = \{y=0\}$ and $W_t = 2\D^{k-1} \times \{|y| < t\}$. Set $\varphi(x,y) :=  \log |y|$. Then $\varphi$ is p.s.h., its Lelong numbers at every point are at most $ 1 $ and  $W_t \subset \{\varphi < \log |t| \} \cap 2\overline{\D^{k}}$. By \eqref{eq:skoda-sublevel}, it follows that the mass of $(\ddc v)^k$ on $W_t$ is bounded by $C_{\varphi,  \gamma} e^{ \gamma \log |t|} = C_\gamma |t|^{\gamma}$
whenever $\gamma < c(v)$. The lemma follows by taking any  $\gamma < c(v)$  and setting $\alpha' = \gamma$.
\end{proof}

Our proof will be by induction on $\ell=1,\ldots,k$. We will need the following lemma.

\begin{lemma} \label{l:exist-I}
Let $\ell \geq 1$ be an integer. There exists a constant $c_\ell>0$ such that if $V_0$ is a hypersurface of degree at most $\delta$  in $\P^\ell$  there is a point $a \in \P^\ell$  such that $\dist(a,V_0)\geq c_\ell \delta^{-1/2}$.
\end{lemma}
\begin{proof}
We denote by  $\volume(B) := \frac{1}{\ell !}  \int_{B} \omegaFS^{\ell}$ the volume of a ball $B \subset \P^\ell$ with respect to the Fubini-Study metric and  by $\volume_{2\ell -2}(V_0 \cap B) := \frac{1}{(\ell -1)!}  \int_{V_0 \cap B} \omegaFS^{\ell - 1}$ the corresponding $(2\ell -2)$-dimensional volume of $V_0 \cap B$. A classical theorem by Lelong \cite{lelong} implies that for every ball $B$ of radius $r$  inside $\P^\ell$ such that $V_0$ intersects $\frac12 B$ we have that $\volume_{2\ell -2}(V_0\cap B) \geq c_\ell' r^{2 \ell - 2}$ for some dimensional constant $c_\ell' >0$.   

Let $c_\ell>0$ and  $\gamma>0$  be two independent constants. Set $N:= \floor{\gamma c_\ell^{-2\ell} \delta^{\ell}}$. Let $B_j$,  $j=1,\ldots, N$ be balls in $\P^\ell$ of radius $2c_\ell \delta^{-1/2}$.  By taking $\gamma$ sufficiently small,  it is not difficult to construct using local coordinates $N$  such balls $B_j$ that are pairwise disjoint.

Denote by $a_j$ the center of $B_j$.  We claim that $\dist(a_j,V_0)\geq c_\ell \delta^{-1/2}$ for at least one $j=1, \ldots, N$, which will prove the lemma. We argue by contradiction.  If that's not the case,  then $V_0$ intersects every ball ${1\over 2}B_j$ of center $a_j$ and of radius $c_\ell \delta^{-1/2}$ for $j=1, \ldots, N$.  Together with Lelong's estimate and the fact that the balls $B_j$ are pairwise disjoint we get that $$\volume_{2\ell -2}(V_0) \geq \sum_{j=1}^N \volume_{2\ell -2}(V_0 \cap B_j)  \geq N  c_\ell' (c_\ell \delta^{-1/2})^{2(\ell-1)} \simeq c_\ell'\gamma c_\ell^{-2} \delta.$$
On the other hand, $\volume_{2\ell -2}(V_0) = \frac{1}{(\ell-1) !} \int_{V_0} \omegaFS^{\ell - 1} =  \frac{1}{(\ell-1) !} \delta$ by definition of degree and the fact that $\omegaFS^{\ell - 1}$ is cohomologous to a line in $\P^\ell$. It follows that $c_\ell'\gamma c_\ell^{-2}$ is bounded  by a constant.  By choosing $c_\ell$ sufficiently small we arrive at a contradiction.  This finishes the proof.
\end{proof}

\noindent \textbf{Central projections and standard boxes.} Let $z = (z_0, \ldots,z_k)$ be the standard euclidean coordinates on $\C^{k+1}$ and $[z] = [z_0: \cdots:z_k]$ be the induced homogeneous coordinates.  Fix $1 \leq \ell \leq k$ and write $z = (z',z'')$ with $z' = (z_0, \ldots,z_\ell) \in \C^{\ell +1}$ and $z'' = (z_{\ell +1}, \ldots,z_k) \in \C^{k - \ell}$.  Let $I_\ell$ be the projective subspace of $\P^k$ of codimension  $\ell +1$ defined by the equation $z'=0$. The {\it central projection} from $I_\ell$ is the map
$$\pi_0: \P^k \setminus I_\ell \longrightarrow \P^\ell, \quad [z] \mapsto [z'].$$
Observe that when  $\ell=k$ we have $I_0=\varnothing$ and $\pi_0:\P^k\to\P^k$ is simply the identity map. 

  Given $a \in \P^\ell$ we will denote by $\pi_a:\P^\ell \setminus \{a\}\to \P^{\ell-1}$ the central projection from $a$. Observe that the restriction of $\pi_a$ to a  hypersurface of $\P^\ell$ of degree $\delta$ not passing through $a$ is a finite holomorphic map of degree $\delta$.

\begin{lemma} \label{l:ramification}
Let $V_0$ be a hypersurface of degree $\delta$ in $\P^\ell$ and fix a point $a \in \P^\ell \setminus V_0$.  Let $\Sigma$ be the set of points $x\in\P^{\ell-1}$ such that $\pi_a^{-1}(x) \cap V_0$ contains less than $\delta$ points. Then $\Sigma$ is a hypersurface of $\P^{\ell-1}$ of degree at most $\delta^2$. 
\end{lemma}
\begin{proof}
Choose a homogeneous coordinate system $z=[z_0:\cdots:z_\ell]$ of $\P^\ell$ such that $a=[0:\cdots:0:1]$. In the affine chart $\{z_0=1\}$, the projection $\pi_a$ is given by 
$\pi_a(w,z_\ell)=w$, where $w=(z_1,\ldots,z_{\ell-1})$. Denote by $a_1(w),\ldots,a_d(w)$  the last affine coordinates of the points in $\pi_a^{-1}(w)\cap V_0$ in the above affine chart,  where each point is repeated according to its multiplicity.
Since $a$ is outside $V_0$, we deduce that $|a_j(w)| = O(\|w\|)$ when $w$ goes to infinity. 
In $\C^{\ell-1}$, the set $\Sigma$ is exactly the zero set of the holomorphic function 
$$\prod_{i<j} (a_i(w)-a_j(w)).$$
This function has a polynomial growth $O(\|w\|^{\delta^2})$ when $w\to\infty$. By Liouville's theorem, it is a polynomial of degree at most $\delta^2$. The lemma follows. 
\end{proof}

We will use in $\P^k$ (resp. $\P^\ell$) the standard Fubini-Study metrics whose associated hermitian form on $\C^{k+1}$ (resp. $\C^{l+1}$) is given by $\ddc \log \|z\|$ (resp. $\ddc \log \|z'\|$). In the proofs below we will consider different coordinate systems on $\C^{k+1}$ related to the standard one by unitary transformations. These coordinate changes preserve the above Fubini-Study metrics. After a unitary change of coordinates mapping a given projective subspace $I$ of $\P^k$ to $I_\ell$,  one can define the central projection from $I$ using the above formula for $\pi_0$.

The lemma below allows us to reduce the proof of Proposition \ref{p:neigh-mes} to the case of hypersurfaces.

\begin{lemma} \label{lemma:codim1}
Let  $V$ be a proper subvariety of $\P^k$ of degree $\leq \delta$. Then $V$ is contained in a hypersurface of $\P^k$ of degree $\leq \delta$.
\end{lemma}

\begin{proof}
Let $s$ be the dimension of $V$ and assume that $s<k-1$. Let $I$ be a generic projective subspace of dimension $k-s-2$ in $\P^k\setminus V$ and let $\pi:\P^k\setminus I\to \P^{s+1}$ be the central projection from $I$. Then $\pi(V)$ is a hypersurface of degree at most $\delta$ in $\P^{s + 1}$. Therefore,   $\overline{\pi^{-1}(\pi(V))}$ is a hypersurface of degree at most $\delta$ containing $V$.
\end{proof}

The following notions will be used in the proof of Proposition \ref{p:neigh-mes}.

\begin{definition} \label{def:type} \rm
Let $1\leq \ell \leq k - 1$. A hypersurface $V$ of $\P^k$ is said to be {\it of type $\ell$} if there exists a central projection $\pi: \P^k \setminus I \to \P^\ell$ from a projective subspace $I$ of codimension $\ell + 1$ and a hypersurface  $V_0 \subset \P^\ell$ such that $V = \overline{\pi^{-1}(V_0)}$.  When $\ell = k$,  we take $I =\varnothing$,  $V_0=V$ and $\pi:\P^k\to\P^k$ to be the identity map. 
\end{definition}

\begin{definition} \rm
 Fix integers $1 \leq \ell \leq k$ and let $r, R_1,R_2$ be positive constants. 
 \begin{itemize} \setlength{\itemindent}{-1em}
 \item A {\it standard box} $\fB$ of size $r$ in $\P^{\ell-1}$ is the polydisc $(r\D)^{\ell-1}$ in the affine chart $\{z_0=1\}$ of $\P^{\ell-1}$ endowed with coordinates $[z_0:\cdots:z_{\ell-1}]$.
 \item A {\it standard box} $\fB'$ of size $(r,R_1)$ in $\P^\ell$ is the polydisc  $(r\D)^{\ell-1}\times (R_1\D)$ in the affine chart $\{z_0=1\}$ of $\P^\ell$ endowed with coordinates $[z_0:\cdots:z_{\ell}]$.
  \item A {\it standard box} $\fB''$ of size $(r,R_1,R_2)$ in $\P^k$ is the polydisc $(r\D)^{\ell-1}\times (R_1\D)\times (R_2\D)^{k-\ell}$ in the affine chart $\{z_0=1\}$ of $\P^k$ endowed with coordinates $[z_0:\cdots:z_{k}]$.
 \end{itemize}

For $\lambda>0$ and $\fB$  a standard box of size $r$ in $\P^{\ell-1}$, we will denote by $\lambda\fB$ the standard box of size $\lambda r$ in $\P^{\ell-1}$ associated to the same coordinates. Similarly for the boxes in $\P^\ell$ and $\P^k$.

Given a standard box $\fB''$ of size $(r,R_1,R_2)$ in $\P^k$  and $r_0 >0$,  the map $$L(z_1,\ldots,z_k)=(r_0 r^{-1} z_1,\ldots, r_0 r^{-1} z_{\ell-1},  r_0 R_1^{-1} z_\ell,  r_0 R_2^{-1} z_{\ell+1}, \ldots,  r_0 R_2^{-1} z_k)$$  defined in  the affine chart $\{z_0=1\} \simeq \C^k$  sends $\fB''$ bijectively to $r_0 \D^k$.  We will call the map $L$ the {\it standard linear isomorphism} between $\fB''$ and $r_0 \D^k$.
\end{definition}

Note that  when we change the coordinates $z$, the standard boxes change accordingly. The following lemma follows from straightforward computations using the above definition for a fixed coordinate system. We leave the details to the reader.  We use here the two central projections $\pi_0: \P^k \setminus I_\ell \to \P^\ell$ and $\pi_a:\P^\ell \setminus \{a\}\to \P^{\ell-1}$ introduced above with $a=[0:\cdots:0:1] \in \P^\ell$. 

\begin{lemma} \label{l:box}
Let $0<r<1$ and let $\fB$ be a standard box of size $r$ in $\P^{\ell-1}$.  Denote by $\fB'$ the standard box of size $(r, 1/r)$ in $\P^\ell$ and 
by $\fB''$ the standard box of size $(r, 1/r, 1/r^2)$ in $\P^k$.   
There are constants $A_0 \geq 1$ and $A_1 \geq 1$ independent of $r$ satisfying the following properties. 
Let $U_a$ be the $A_0r$-neighborhood of $a$ in $\P^\ell$ and $U_{I_\ell}$ be the $A_0r$-neighbourhood of $I_\ell$ in $\P^k$.  
Then
\begin{enumerate}
\item The standard box of size $(10r,1/r)$ in $\P^\ell$ contains the set $\pi_a^{-1}(10\fB)\setminus U_a$;
\item The  standard box  of size $(10r,10/r,1/r^2)$ in $\P^k$ contains the set $\pi_0^{-1}(10\fB')\setminus U_{I_\ell}$;
\item Let  $L:10 \fB'' \to 10 \D^k$ be the standard linear isomorphism.  Then the differential $dL$ satisfies $\|dL\|\leq A_1/r^2$ on $10 \fB''$ and $\|(dL)^{-1}\|\leq A_1/r^2$ on $10 \D^k$.  Here we use the Fubini-Study metric on $10 \fB''$ and the Euclidean metric on $10 \D^k$. 
 \end{enumerate}
\end{lemma}

We can now  prove  the proposition.

\begin{proof}[Proof of Proposition \ref{p:neigh-mes}]
By Lemma \ref{lemma:codim1} we can assume that $V$ is a hypersurface of $\P^k$ of degree $ \leq \delta$.  One can assume that $\delta \geq 2$ since the case $\delta <2$ is an immediate consequence of the case $\delta =2$.   Let $1 \leq \ell \leq k$ be the type of $V$ as in Definition \ref{def:type}.  We shall proceed by induction on $\ell$.

When $\ell = 1$ we have that  $V = \overline{\pi^{-1}(V_0)}$ where  $\pi: \P^k \setminus I \to \P^1$ is the central projection from a projective subspace $I$ of codimension $2$ and $V_0$ is a set of at most $\delta$ points in $\P^1$.  Hence $V$ is a union of at most $\delta$ hyperplanes in $\P^k$.  In this case the result follows directly from Lemma \ref{l:H-neigh}.

Consider now $2 \leq \ell \leq k$ and assume that the result has been proven for hypersurfaces of type $\ell-1$.  Let $V$ be a hypersurface of type $\ell$ and of degree $\leq\delta$. For simplicity, adding some hyperplanes to $V$ allows us to assume that $\deg(V)=\delta$. By definition, $V = \overline{\pi^{-1}(V_0)}$ for a   hypersurface  $V_0 \subset \P^\ell$ of degree $\delta$ and  a central projection $\pi: \P^k \setminus I \to \P^\ell$.  After a unitary change of coordinates we can assume that $I = I_\ell$ and $\pi = \pi_0$ is the canonical  projection defined above.

By Lemma \ref{l:exist-I}, we can find a point $a\in \P^\ell$ such that $\dist(a_j,V_0)\geq c_\ell \delta^{-1/2}$.  Consider the central projection $\pi_a:\P^\ell\setminus\{a\}\to \P^{\ell-1}$ and let  $\Sigma$ be as in Lemma \ref{l:ramification}.  By that lemma, this is a hypersurface of degree at most $\delta^2$ in $\P^{\ell-1}$.  Fix two large constants $\kappa_1 \gg \kappa_0 \gg 1$ and set
  \noindent
\begin{minipage}[t]{0.52\textwidth}  
  \begin{itemize}
    \item $U_\Sigma := \Tub(\overline{\pi_0^{-1}(\pi_a^{-1}(\Sigma))}; \delta^{-\kappa_0})$ in $\P^k$,
    \item $W := \Tub(\Sigma; 100k \delta ^{-\kappa_1})$ in $\P^{\ell-1}$,
  \end{itemize}
\end{minipage}
\begin{minipage}[t]{0.45\textwidth}
  \begin{itemize}
    \item $U_a := \Tub(a; \delta^{-\kappa_0})$ in $\P^\ell$,  
    \item  $U_I:= \Tub(I; \delta^{-\kappa_0})$ in $\P^k$.
  \end{itemize}
\end{minipage}

Using that $\kappa_0\ll\kappa_1$,  it is not difficult to see that $\pi_0^{-1}(\pi_a^{-1}(W))\subset U_\Sigma$.  Notice that $\overline{\pi_0^{-1}(a)}$ is a projective subspace of $\P^k$.  
By induction hypothesis applied to $\pi_0^{-1}(\pi_a^{-1}(\Sigma))$ and $2\delta^2$ instead of $\delta$,  we have $\nu(U_\Sigma)\leq   {1\over 2}\delta^{-2}$ because $\overline{\pi_0^{-1}(\pi_a^{-1}(\Sigma))}$ is of type $\ell-1$.  Therefore, it remains to show that 
\begin{equation} \label{eq:measure-tube}
\nu\big( \Tub(V;\delta^{-\kappa}) \setminus U_\Sigma \big)\leq {1\over 2} \delta^{-1}
\end{equation}
 provided that $\kappa$ is large enough.

Cover $\P^{\ell-1}\setminus W$ by $N$ standard boxes $\fB_j$ of size $\delta^{-\kappa_1}$ whose centers are outside $W$.  These boxes may correspond to different coordinate systems.  The number $N$ of boxes can be chosen to be a large power of $\delta$.   In particular,  it grows polynomially with $\delta$. Denote by $\fB_j'$ the box of size $(\delta^{-\kappa_1}, \delta^{\kappa_1})$ in $\P^\ell$ and $\fB_j''$ the box of size $(\delta^{-\kappa_1}, \delta^{\kappa_1}, \delta^{2\kappa_1})$ in $\P^k$ associated to $\fB_j$.  For each  $j$, we use an adapted coordinate system as above.  Using Lemma \ref{l:box}  we can ensure the following

\begin{enumerate}
\item The box of size $(10 \delta^{-\kappa_1}, \delta^{\kappa_1})$ in $\P^\ell$ defined in the same affine chart as  $\fB_j'$ contains $\pi_a^{-1}(10 \fB_j)\setminus U_a$; in particular, it contains $V_0\cap \pi_a^{-1}(10 \fB_j)$;
\item The box of size $(10 \delta^{-\kappa_1}, 10 \delta^{\kappa_1}, \delta^{2\kappa_1})$ in $\P^k$  defined in the same affine chart as  $\fB_j''$ contains $\pi_0^{-1}(10 \fB_j')\setminus U_I$,  hence it contains $\pi_0^{-1}(10 \fB_j')\setminus U_\Sigma$.
\end{enumerate}

\noindent  \underline{\emph{Claim}}: For every $j=1,\ldots,N$,   $V_0\cap \pi_a^{-1}(10\fB_j)$ is a union of $\delta$ disjoint connected components $\Lambda_{j,p}$,  $1\leq p \leq \delta$.  Moreover,  each $\Lambda_{j,p}$ is  the graph of a holomorphic function over  $10\fB_j$ whose modulus is bounded by $\delta^{\kappa_1}$. 

\smallskip

\proof[\underline{\emph{Proof of claim}}] By construction,  each box $10 \fB_j$ is disjoint from $\Sigma$.  In particular, $\pi_a|_{V_0}$ is unramified over $10 \fB_j$,  so that  $V_0\cap \pi_a^{-1}(10\fB_j)$ is a union of $\delta$ disjoint connected components $\Lambda_{j,p}$,  $1\leq p\leq \delta$.  By our choice of $a$ and  $\kappa_1 \gg \kappa_0 \gg 1$,  we see that $U_a$ is disjoint from $V_0$.  We deduce from Property (1) above that $\Lambda_{j,p}$ is contained in the standard box of size $(10 \delta^{-\kappa_1}, \delta^{\kappa_1})$ in $\P^\ell$ in the same affine chart as  $\fB_j'$  and $\pi_a:\Lambda_{j,p}\to 10\fB_j$ is bijective.  The claim follows.  
\endproof

Since the boxes $\fB_j$  cover  $\P^{\ell-1}\setminus W$,  the above graphs $\Lambda_{j,p}$ restricted to $\fB_j'$ cover $V_0\setminus \pi_a^{-1}(W)$.  Now set  $\Gamma_{j,p}:=\pi_0^{-1}(\Gamma_{j,p})\cap (10\fB_j'')$. Then  $\Gamma_{j,p}$ can be seen as the graph of a holomorphic function over $(10\fB_j)\times (10 \delta^{2\kappa_1}\D)^{k-\ell}$.  
More precisely,  if we work on the chart $\{z_0=1\}$ as above, then $\Gamma_{j,p}$ is a graph with respect to the projection $(z_1,\ldots,z_k)\mapsto (z_1,\ldots,z_{\ell-1},z_{\ell+1},\ldots,z_k)$.  Moreover,  these graphs are contained  in the set $\{|z_\ell|\leq \delta^{\kappa_1}\}$.  

From Property (2) above,  the sets $\Gamma_{j,p}\cap \fB_j''$ cover $V\setminus U_\Sigma$.  Therefore,  the $\delta^{-\kappa}$-neighbourhoods of $\Gamma_{j,p}\cap \fB_j''$ cover the $\delta^{-\kappa}$-neighbourhood of $V\setminus U_\Sigma$.  Since the total number of graphs is polynomial in $\delta$,  the desired estimate \eqref{eq:measure-tube} will follow once we show that,  given $\kappa_2 >0$ large enough,  there exists $\kappa \geq 1$ such that
\begin{equation} \label{eq:measure-tube2}
\nu(\Tub(\Gamma_{j,p}\cap \fB_j''; \delta^{-\kappa})) \leq  C  \delta^{-\kappa_2} \quad \text{ for all } 1 \leq j \leq N,  \,  1\leq p \leq \delta 
\end{equation}
for some constant $C >0$. 

In order to show \eqref{eq:measure-tube2} we will combine Lemmas  \ref{l:Gamma-neigh} and  \ref{l:box}.  To simplify the notation we fix $j$ and $p$ and  denote $\fB:=\fB_{j}$, $\fB':=\fB_{j}'$, $\fB'':=\fB_{j}''$ and $\Gamma:=\Gamma_{j,p}$.  Denote also by $\Upsilon(\delta^{-\kappa}) := \Tub(\Gamma \cap \fB''; \delta^{-\kappa})$.  Let $L$ be the standard linear isomorphism between $10\fB''$ and $10 \D^k$.  Define 
$$\widetilde\Gamma:=L(\Gamma),  \,\,   \widetilde \Upsilon(\delta^{-\kappa}) :=L(\Upsilon(\delta^{-\kappa}) ) \,\,  \text{ and } \,\,   \widetilde\nu:=L_*(\nu).$$  Since $\nu(\Upsilon(\delta^{-\kappa}) )=\widetilde\nu(\widetilde{\Upsilon}(\delta^{-\kappa}))$, we have to show that given $\kappa_2 >0$,  there exists $\kappa \geq 1$ such that
\begin{equation} \label{eq:measure-tube3}
\widetilde\nu(\widetilde{\Upsilon}(\delta^{-\kappa}))\leq C \delta^{-\kappa_2}.
 \end{equation}
 
 Recall that $\fB''$ is of size $(\delta^{-\kappa_1}, \delta^{\kappa_1}, \delta^{2\kappa_1})$, so $\|dL\|\leq C_1 \delta^{2\kappa_1}$ by Lemma \ref{l:box}.  Using the notation of Lemma \ref{l:Gamma-neigh},  we deduce that $\widetilde{\Upsilon}(\delta^{-\kappa})$ is contained in $W_t$ for $t = C_2 \delta^{-\kappa+2\kappa_1}$ for some $C_2 > 0$.  For simplicity, we set $\Upsilon^*(\delta^{-\kappa+2\kappa_1}) := W_t$ for $t$ as above.   
  
  Since  $\widetilde \Gamma$ is the image of $\Gamma$ under $L$,  it follows from the above discussion that $\widetilde \Gamma \subset 10 \D^k$ is the graph of a holomorphic function $h$ over $10 \D^{k-1}$ with values in $\overline\D$. In particular,  we can apply Lemma \ref{l:Gamma-neigh}.  Let $v$ be a local potential of $\nu$ on $\fB''$ and $\widetilde v = v \circ L^{-1}$ be the corresponding local potential of $\widetilde \nu$ on $10 \D^k$.  Since by assumption the global potential of $\nu$ is $\alpha$-H\"older continuous for some $0< \alpha \leq 1$,  $v$ is $(A,\alpha)$-H\"older continuous for some $A >0$.  Because $\|dL^{-1}\|\leq C_1 \delta^{2\kappa_1}$, we have that $\widetilde v$ is $(A' \delta^{2\kappa_1}, \alpha)$-H\"older continuous for some $A' >0$.  Finally,  we set $\widehat v =   \delta^{-2 \kappa_1} \widetilde v$ so that $\widehat v$ is $(A',\alpha)$-H\"older continuous.  Recall that $\widetilde{\Upsilon}(\delta^{-\kappa}) \subset W_t = \Upsilon^*(\delta^{-\kappa+2\kappa_1})$.  By Lemma \ref{l:Gamma-neigh} there exist constants $A''>0$ and $\alpha' >0$ such that
 $$ (\ddc \widehat v )^k (\Upsilon^*(\delta^{-\kappa+2\kappa_1})) \leq A'' \delta^{-\alpha' \kappa + 2\alpha' \kappa_1}.$$
  
 Now,  $\widetilde \nu = (\ddc \widetilde v )^k = \delta^{2k \kappa_1} (\ddc \widehat v)^k$,  so we conclude that
 $$\widetilde\nu(\widetilde{\Upsilon}(\delta^{-\kappa}))\leq  \widetilde\nu(\Upsilon^*(\delta^{-\kappa+2\kappa_1})) \leq  A''\delta^{-\alpha' \kappa + 2\alpha' \kappa_1 + 2k \kappa_1}. $$
By taking $\kappa$ large enough yields \eqref{eq:measure-tube3} and finishes the proof.
\end{proof}

\section{Local coordinates, Manhattans and inverse branches} \label{s:branch}
In this section,  we build a good covering of $\P^k$ over which we will later construct the inverse branches of $f^n$.

Denote by $\W$ the open square $(-1,1)^2$ of side $2$ in $\C\simeq\R^2$ and,  for $r>0$, denote by $r\W$ the image of $\W$ under the linear map $z\mapsto rz$. We cover $\P^k$ by a finite number of charts $(\Omega'_j,\pi_j)$, $1\leq j\leq M$, such that 
\begin{enumerate}\setlength{\itemindent}{-7pt}
\item $\Omega'_j$ is an open set of $\P^k$ and $\pi_j:\Omega'_j\to 100 \W^k$ is a biholomorphic map. For simplicity,  we can take $\pi_j$ to be the restriction of some standard affine chart to $\Omega'_j$; 
\item There exists a constant $A_2 > 0$ such that $\|d\pi_j\|\leq A_2$ and $\|(d\pi_j)^{-1}\|\leq A_2$ for every $1\leq j\leq M$,  where the norms are with respect to the Fubini-Study metric on $\P^k$ and the Euclidean metric on $\C^k$;
\item Let $\Omega_j:=\pi_j^{-1}(\W^k)$  and ${1\over 10} \Omega_j:=\pi_j^{-1}({1\over 10}\W^k)$. Then,  the sets ${1\over 10} \Omega_j$ cover $\P^k$.
\end{enumerate}

From now on, we fix this atlas.  A ball will mean a ball with respect to these charts and they will often  be denoted by $\B$.  Since the number of charts is finite,  the choice of chart won't affect our estimates. Condition (2) above shows that the Euclidean metrics on charts are comparable with the Fubini-Study metric.  In particular,  they are comparable on the overlaps of the charts.

We introduce now the notion of Manhattans, cells and street networks.  These cells will be the open sets over which we will consider our inverse branches.  Their special geometry will allow us to provide good estimates.  We first define them over $\C^k$.  The definition is transferable  to $\P^k$ via the above atlas.

\begin{definition}[Manhattans,  cells and  street networks]  \rm
Identify $\C^k$ with $\R^{2k}$ in the standard way.  Fix a constant  $0<r<1/100$ and a vector $\tau$ in $\C^k\simeq\R^{2k}$.  A \textit{cell} is an open cube of the form
$$\W_{r,\eta}^k:=r\W^k + \tau + 2r \eta \quad \text{with} \quad \eta\in\Z^{2k}.$$ 
In order to simplify the notation we do not include $\tau$ in the index. Notice that the above cells are disjoint.   Moreover,  the closed cubes $\overline\W_{r,\eta}^k$ cover $\C^k$ and two of them can only overlap along a common face of lower dimension.

A  \textit{street network} is the complement of the union of the cubes 
$$(1-r)\W_{r,\eta}^k:= (1-r) r\W^k + \tau + 2r \eta,$$
where $\eta$ runs through $\Z^{2k}$.  An  \textit{extended street network},  denoted by  $q\bbS_r$ for $q=2,3,4$  is the complement of the union of the cubes
$$(1-qr)\W_{r,\eta}^k:= (1-qr) r\W^k + \tau + 2r \eta,$$
where $\eta$ runs through $\Z^{2k}$.  By {\it center} of $\W_{r,\eta}^k$ and $(1-qr)\W_{r,\eta}^k$ we mean the point $\tau + 2r \eta$.

A \textit{Manhattan} of parameters $r,\tau$ is the data of an open subset $\Omega$ of $\C^k$ (or a chart in $\P^k$) together with the cells inside $\Omega$ and the (extended) street networks restricted to $\Omega$.
\end{definition}

We will only use Manhattans for the charts $\Omega_j:=\pi_j^{-1}(\W^k)$ of $\P^k$ described at the beginning of the section with small parameters $r$,  where we use $\pi_j$ to identify $\Omega_j$ with an open cube of $\C^k$.  We will denote by $\W_{r,\eta}^{k,j}$, $\bbS_r^j$ and $q\bbS_r^j$ the corresponding cells and (extended) street networks.  It is not difficult to see that the closed cells inside $\Omega_j$ cover ${1\over 2}\Omega_j$ when $r<1/100$.

A key point of our proof is to construct good inverse branches for $f^n$ on the cells of a suitable Manhattan.    The main obstruction to obtain these inverse branches is the postcritical set of $f$.   Denote by $\rmC_f$ the critical set of $f$ and define {\it the postcritical set} $\PC_m$ of order $m$ of $f$ by
$$\PC_0=\varnothing, \quad \PC_m:=f(\rmC_f)\cup f^2(\rmC_f)\cup \cdots \cup f^m(\rmC_f) \quad \text{and}\quad \PC_\infty:=\cup_{m\geq 1}f^m(\rmC_f).$$
Here, we drop the index $f$ from $\PC$ for simplicity.  The set $\PC_m$ is a hypersurface in $\P^k$ whose degree is  
bounded by a constant times $d^{(k-1)m}$.

\begin{lemma} \label{l:Manhattan}
Let $0<r<1/100$ be any constant and $1\leq j\leq M$.  
Then there is a vector $\tau \in \C^k$, which may depend on $j$, such that $\mu(4\bbS_r^j\cap\Omega_j)\leq 30kr$.
Moreover, the centers of the cells of the corresponding  Manhattans of parameters $r,\tau$  are all outside $\PC_\infty$.
We say that such a Manhattan is good. 
\end{lemma}
\begin{proof}
Here,  we identify $\C^k$ with $\R^{2k}$ and use its standard coordinate system $(x_1,\ldots,x_{2k})$. 
Consider $\tau^*=(a_1^*,\ldots, a_{2k}^*)$ with $0\leq a_l^*\leq r^2$ and
$\tau=\tau^*+(a_1,\ldots,a_{2k})$ with $a_l\in \{0,9r^2, 18r^2,\ldots, 9(N-1)r^2\}$ and $N:=\floor{1/(10r)}-1$. Fix $\tau^*$ so that for all choices of $a_l$ the cells of the Manhattan associated to $r,\tau$ have centers outside $\PC_\infty$. This property holds for any generic choice of $\tau^*$ with respect to the Lebesgue measure because $\PC_\infty$ has Lebesgue measure zero,  being  a countable union of hypersurfaces.  We show that one of the choices of $a_l$ as above  satisfies the lemma.

  Observe that $4\bbS_r^j$   is covered by the union of the set 
$$\bbS(a_1):=\{|x_1-a_1^*-a_1-r|\leq 4r^2\} + 2r\Z^{2k}$$
and the $2k-1$ other similar sets associated to the coordinates $x_2,\ldots,x_{2k}$.
Therefore, it suffices to show that for some choice of $a_1$ we have $\mu(\bbS(a_1))\leq 1/N$.  Note that  the $N$ sets $\bbS(9lr^2)$ with $l=0,\ldots,N-1$ are disjoint.  Since the total mass of $\mu$ on $\Omega_j$ is at most $1$,  we must have  $\mu(\bbS(a_1))\leq 1/N$ for some $a_1$ as above.  This finishes the proof of the lemma. 
\end{proof}

In what follows, we will only use good Manhattans for the sets $\Omega_j$. 

\begin{definition} \rm \label{d:good_branch}
Let $D\subset\P^k$ be a connected set.  A continuous map $g:D\to\P^k$ is called an  {\it  inverse branch of order $m$} for $f$ on $D$  if  $f^m \circ g$ is the identity map on  $D$.  We call $\diam(g(D))$ {\it the diameter} of the inverse branch.
\end{definition}

We will mostly work in the case where $D$ is a complex manifold.  In this case, every inverse branch $g$ on $D$ is automatically holomorphic.  When $D$ is a domain in $\P^k$,  observe that different inverse branches of  a given order $m$ on $D$ have disjoint images. In this case,  for any point $a\in D$,  the branch $g$ is uniquely determined by the point $g(a)\in f^{-m}(a)$ and there are at most $d^{km}$ such inverse branches for $D$.

Later on,  we will apply the next proposition to $\ell \simeq \xi m$ for some fixed $0<\xi<1$.
\begin{proposition} \label{p:inverse}
There are constants $A_3>0$ and $0 < \vartheta_0 < 1$ such that the following holds for  every $m \geq 0$ and every  $0\leq \ell \leq m$. If $\B$ is a ball of $\P^k$ such that $\B\cap \PC_{\ell}=\varnothing$, then the ball $\vartheta_0\B$ admits at least $(1-A_3 d^{-\ell}) d^{km}$ inverse branches of order $m$ and of diameters less than $A_3 d^{-\frac{m-\ell}{2}}$.  
\end{proposition}
\begin{proof}[Sketch of the proof]
The proof follows from the arguments in  \cite[Proposition 1.51]{dinh-sibony:cime}.  Since $\B\cap \PC_{\ell}=\varnothing$,   the number of inverse branches of order $\ell$ defined over $\B$ is maximal,  namely $d^{k\ell}$.  Let $a$ be the center of $\B$ and denote by $\Dc \simeq \P^{k-1}$ the family of complex lines through $a$.  Let $\Delta \in \Dc$.  By B\'ezout's theorem $f^{-\ell}(\Delta) \cap f(\rm C)$ contains at most $d^{(k-1)\ell} d^{k-1} \deg({\rm C}_f)$ points.  In other words, there are at most $d^{(k-1)\ell} d^{k-1} \deg({\rm C}_f)$ inverse branches on $\Delta \cap \B$ that meet $\PC_1$.  We discard such branches.  By construction,  the remaining $d^{k\ell}(1 - d^{-\ell}d^{k-1} \deg({\rm C}_f))$ branches can be pulled back by $f$,  providing $d^{k(\ell+1)}(1-d^{-\ell}d^{k-1} \deg({\rm C}_f))$ branches of order $\ell + 1$.  Arguing inductively,  one can show that there is at least $d^{km}(1-A d^{-\ell}) $ inverse branches of order $m$ on $\Delta \cap \B$ for some constant $A>0$.

Using the fact that the total area of $f^{-m}(\Delta)$ is $d^{(k-1) m}$ one can exclude a small number of branches of large area,  yielding  $d^{km}(1-A' d^{-\ell}) $ inverse branches of order $m$ on $\Delta \cap \B$ whose image has area less than $d^{-(m-\ell)}$ for some constant $A'>0$.  Then,  the restrictions of these branches to $\Delta \cap \frac12 \B$  are of diameter of order  $d^{-\frac{m-\ell}{2}}$  (see  \cite[Lemma 1.55]{dinh-sibony:cime}).

Applying the above argument for a large enough set of lines $\Delta \in \Dc$ and using Sibony-Wong's theorem \cite{sibony-wong},  one can show that the above diameter estimates is valid on the whole ball  $\vartheta_0\B$ for some constant $\vartheta_0 > 0 $.  
\end{proof}

The following equidistribution theorem,  which is a direct consequence of results in \cite{dinh-sibony:equid-speed},  will be crucial to us.

\begin{theorem} \label{t:eq-pt}
There exist an integer $n_0\geq 1$ and a constant $A_4>0$ such that,  for every point $a$ in $\P^k\setminus\PC_{n_0}$, every test $\Cc^1$ function $\phi$ on $\P^k$ and every $m\geq 0$, we have 
$$\Big| \Big\langle {1\over d^{km}} (f^m)^*(\delta_a)-\mu,\phi \Big\rangle \Big| \leq A_4 \Big[1+\log^+{1\over \dist(a,\PC_{n_0})}\Big]d^{-{m\over 3}} \|\phi\|_{\Cc^1},$$
where $\log^+:=\max(\log,0)$. 
\end{theorem}

Consider a good Manhattan of parameters $r,\tau$ for a chart $\Omega_j:=\pi_j^{-1}(\W^k)$ of $\P^k$.  We will only need the case where $r$ is independent of $j$ but $\tau$ may depend on $j$.
As before, define ${1\over 2}\Omega_j:= \pi_j^{-1}({1\over 2}\W^k)$. We will add the index $j$ to the cells  and street networks to highlight the dependence on $j$.
The following corollary will later be applied to the centers of cells of a suitable Manhattan,  where $n_0$ is as above.  We will take $m=\floor{\zeta n}$ or $m=\ceil{(1-\zeta)n}$ for some large integer $n$ and a small constant $0<\zeta<1$. 

\begin{corollary} \label{c:orbit-street}
Let $0<r<1/100$ and $n_0$ be as in Theorem \ref{t:eq-pt}.  There is a constant $A_5>0$ independent of $r$ such that for every $m\geq 0$ with $d^{-m}<r^{15}$, every $1\leq j\leq M$, and every point $a$ in $\P^k$ with $\dist(a,\PC_{n_0})\geq r$, the number of points in $f^{-m}(a)\cap 3\bbS_r^{j}\cap {1\over 2}\Omega_j$, counted with multiplicities, is bounded by $A_5rd^{km}$.
\end{corollary}
\begin{proof}
One needs to bound the mass of the measure $(f^m)^*(\delta_a)$ on $3\bbS_r^{j}\cap {1\over 2}\Omega_j$ by a constant times $rd^{km}$. For this purpose, we apply Theorem \ref{t:eq-pt} for a suitable function $\phi$. 

Choose a smooth cut-off function $0\leq \chi \leq 1$ supported by $\Omega_j$ and equal to 1 on ${1\over 2}\Omega_j$.
For each $\eta\in\Z^{2k}$, choose a smooth cut-off function $0\leq \chi_\eta \leq 1$ with support in $(1-3r)\W_{r,\eta}^{k,j}$ and equal to 1 on $(1-4r)\W_{r,\eta}^{k,j}$.  By constructing $\chi_\eta$ as a product of cutoff functions on one-dimensional cells,  one can ensure  that $\|\chi_\eta\|_{\Cc^1}\lesssim 1/r^2$.  Define $\phi:=\chi(1-\sum_\eta \chi_\eta)$. It is not difficult to see that $\|\phi\|_{\Cc^1}\lesssim 1/r^2$ because the supports of $\chi_\eta$ are disjoint. By construction, $\supp(\phi)\subset 4 \bbS_r^{j}\cap\Omega_j$ and $\phi=1$ on $3\bbS_r^{j}\cap {1\over 2}\Omega_j$.
We deduce from Theorem \ref{t:eq-pt}  that 
$$ \langle (f^m)^*(\delta_a),\phi \rangle \leq d^{km}\langle\mu,\phi\rangle + A  \Big[1+\log^+{1\over \dist(a,\PC_{n_0})}\Big]d^{-{m\over 3}} r^{-2} d^{km},$$
for some constant $A>0$. The left hand side is larger or equal to the mass of $(f^m)^*(\delta_a)$ on $3\bbS_r^{j}\cap {1\over 2}\Omega_j$.  Moreover, $\langle\mu,\phi\rangle$ is bounded by the mass of $\mu$ on $4 \bbS_r^{j}\cap\Omega_j$ which is bounded by a constant times $r$ by Lemma \ref{l:Manhattan}. The result follows from the properties of $a$, $r$ and $m$.
\end{proof}

\begin{corollary} \label{c:orbit-cell}
There is a constant $A_6>0$ such that for every  $1\leq j \leq M$,  every cell  $\W_{r,\eta}^{k,j}$ in $\Omega_j$,  every $m\geq 0$ with $d^{-m}<r^{30k}$, and every point $a$ in $\P^k$ with $\dist(a,\PC_{n_0})\geq r$, if $p$ is the number of points in $f^{-m}(a)\cap (1-2r)\W_{r,\eta}^{k,j}$, counted with multiplicities, then it satisfies 
$$\mu\big((1-3r)\W_{r,\eta}^{k,j}\big) d^{km} -A_6r^{2k+2}d^{km} \leq p \leq \mu\big(\W_{r,\eta}^{k,j}\big) d^{km} +A_6r^{2k+2}d^{km}.$$
\end{corollary}
\begin{proof}
Observe that $p$ is the mass of $(f^m)^*(\delta_a)$ in $(1-2r)\W_{r,\eta}^{k,j}$. 
We choose a smooth cut-off function $0\leq \phi\leq 1$ with compact support in $(1-2r)\W_{r,\eta}^{k,j}$ and equal to 1 on $(1-3r)\W_{r,\eta}^{k,j}$ such that $\|\phi\|_{\Cc^1} \lesssim r^{-2}$. 
By Theorem \ref{t:eq-pt}, we have 
$$ p-  \langle \mu, \phi\rangle d^{km} \geq \langle (f^m)^*(\delta_a), \phi\rangle - \langle \mu, \phi\rangle d^{km}\gtrsim - d^{-{m\over 3}} |\log r| r^{-2} d^{km} \gtrsim - r^{2k+2}d^{km}.$$
The first inequality in the statement follows from the fact that $ \langle \mu, \phi\rangle \geq \mu\big((1-3r)\W_{r,\eta}^{k,j}\big)$. The second one is obtained in the same way using a cut-off function  with 
compact support in $\W_{r,\eta}^{k,j}$ and equal to 1 on $(1-r)\W_{r,\eta}^{k,j}$.
\end{proof}

Note that the total number of cells is $O(r^{-2k})$.  Therefore, the quantity $r^{2k+2}$ in the last corollary is very small.  This  good control is only possible thanks to Theorem \ref{t:eq-pt}.

\section{Construction of repelling periodic points} \label{s:periodic}

The aim of this section is to construct a good number of repelling periodic points together with diameter estimates.  We keep the notation of the last section.  We first fix an arbitrary index $1\leq j_0\leq M$ and define $\Omega:=\Omega_{j_0}=\pi_{j_0}^{-1}(\W^k)$. For simplicity, we will only construct repelling periodic points in ${1\over 2}\Omega$ but the construction uses the other charts ${1\over 2}\Omega_j$ as well.  

Let $0<\gamma<1$ be as in the statement of Theorem \ref{t:main} and $\kappa\geq 1$ be as in Proposition \ref{p:neigh-mes}. Let $0 < \vartheta_0 < 1$  be the constant from Proposition \ref{p:inverse}.  Fix a constant $0<\zeta<\gamma/4$. 
Fix also a constant $\gamma_0 >0$ such that $800\gamma_0\kappa k^2 <\zeta$ and define $\gamma_1:=20\gamma_0\kappa k$. Observe that 
$\gamma_0<1/3200$ and $20\gamma_0 < \gamma_1<1/(80k)$.
In what follows, we will work with inverse branches of orders $n, \ceil{(1-\zeta)n}$ or $\floor{\zeta n}$ for given cells of a good Manhattan.  For this purpose,  the Manhattans and many of the objects related to them will depend on $n$.  We only need to consider $n$ large enough.
Here and in what follows,  we use this property in order to absorb some constants and simplify the notation.  
In particular, we will need that $100k\vartheta_0^{-1}d^{-\gamma_1 n} \leq \delta^{-\kappa}$ for $\delta$ defined below.

Denote
$$V:=\PC_{\floor{10\gamma_0 n}},  \quad  U := \Tub(V, 10k \vartheta_0^{-1}d^{-\gamma_1 n}) \quad \text{and} \quad U':= \Tub(V,   100k \vartheta_0^{-1}d^{-\gamma_1 n}).$$
Observe that  $V$ is a hypersurface of degree bounded by a constant times $d^{10\gamma_0 (k-1) n}$.  
By Proposition \ref{p:neigh-mes} applied to $\delta:=d^{\floor{10\gamma_0 kn}}$, we obtain that 
$$\mu(U)\leq \mu(U')\leq d^{-\floor{10\gamma_0 kn}}.$$

\smallskip

Fix a good Manhattan for each domain $\Omega_j$ with $$r:=d^{-\gamma_1 n}$$ as in Lemma \ref{l:Manhattan} including one for $\Omega=\Omega_{j_0}$. 
We will use the notation introduced in the last section, in particular, we will add an index $j$ to objects related to $\Omega_j$ while the ones associated to $\Omega$ have no such an index. The cells inside the charts ${1\over 2}\Omega_j$ cover $\P^k$ and we will only consider such cells which are moreover admissible in the following sense. Note that thanks to  the choice of $r$, these cells cover $\P^k\setminus U'$. 

\begin{definition} \rm \label{d:good-cell}
We keep the notation of Section \ref{s:branch}.
\begin{enumerate}
\item A cell $\W_{r,\eta}^{k,j}$  is {\it admissible} if its center is outside $U$. 
\item Let $\W_{r,\eta}^{k}$ and $\W_{r,\eta'}^{k,j}$ be admissible cells. We say that $\W_{r,\eta'}^{k,j}$ is {\it nice with} $\W_{r,\eta}^{k}$ if it admits at least 
$$q_{r,\eta}:=\big[\mu\big((1-3r)\W_{r,\eta}^{k}\big) -d^{-2\gamma_0 n}\mu\big(\W_{r,\eta}^{k}\big) - r^{2k+1}\big] d^{k\ceil{(1-\zeta)n}} $$ 
inverse branches of order $\ceil{(1-\zeta)n}$ with images in $(1-r)\W_{r,\eta}^{k}$ and with diameters less than $d^{-{1-2\zeta\over 2}n}$.
\item An admissible cell $\W_{r,\eta}^{k}$ is said to be {\it safe} if the cells $\W_{r,\eta'}^{k,j}$ which are not nice with $\W_{r,\eta}^{k}$, have a total $\mu$ measure not more that $d^{-\gamma_0 n}$. Otherwise, we call  it \textit{unsafe}.  If $\W_{r,\eta}^{k}$ is not admissible,  we shall also call it {\it unsafe}. 
\end{enumerate}
\end{definition}

Note that we only define  safe cells for the fixed Manhattan on ${1\over 2}\Omega$. Recall that all considered cells are inside the charts ${1\over 2}\Omega_j$ and the charts ${1\over 10} \Omega_j$ cover $\P^k$.

\begin{proposition} \label{p:unsafe}
There is a constant $A_7>0$ such that the unsafe cells in ${1\over 2}\Omega$ have a total $\mu$ measure not more than $A_7d^{-2\gamma_0 n}$.  
\end{proposition}
\begin{proof}
Recall that we only need to consider $n$ large enough and often use this property to absorb some constants and simplify the notation.
Since the non admissible cells  are contained in $U'$, their total $\mu$ measure is bounded by a constant times $d^{-10\gamma_0 n}$. Thus, we only need to consider the admissible ones.
For $1\leq j\leq M$,  we say that an admissible cell $\W_{r,\eta}^{k}$ in $\frac12 \Omega$ is {\it $j$-safe} if the union of all cells in $\frac12 \Omega_j$  which are not nice with $\W_{r,\eta}^{k}$ have a  $\mu$ measure bounded by $d^{-2\gamma_0 n}$.  Notice that there is here a factor 2 in the power.
We want to show that the union of all $j$-unsafe cells are of measure bounded by a constant times $d^{-2\gamma_0 n}$.   Since we have a fixed number $M$ of charts $\Omega_j$,  it is enough to prove the last estimate for a fixed index $j$.

From now on, we fix $1\leq j\leq M$.  In order to count the inverse branches, it is convenient to denote by $\fW_1,\ldots,\fW_N$ the admissible cells of $\frac12 \Omega$,  and $m_1,\ldots,m_N$ their $\mu$ measures.
Denote by $\fW_1',\ldots,\fW_{N'}'$ the admissible cells of $\frac12 \Omega_j$, $a_1',\ldots,a_{N'}'$ their centers and $m_1',\ldots,m_{N'}'$ their $\mu$ measures.

Define $\I:=\{1,\ldots,N\}$ and $\I':=\{1,\ldots,N'\}$.
Consider the function $\sigma:\I\times \I'\to \N$ with $\sigma(s,s')$ the number of points $b$ in $f^{-\ceil{(1-\zeta)n}}(a_{s'}')\cap \fW_s$ such that 
$b$ is not associated to an inverse branch of order $\ceil{(1-\zeta)n}$ on $\fW'_{s'}$ which has a diameter smaller than $d^{-{1-2\zeta\over 2}n}$. 

We apply Proposition \ref{p:inverse} to $\floor{10\gamma_0 n}$ and $\ceil{(1-\zeta)n}$ instead of $\ell$ and $m$.  We also choose $\B$ of center $a'_{s'}$ and radius $k \vartheta_0^{-1} r$ so that $\vartheta_0 \B$ contains $\fW'_{s'}$. Observe that the inverse branches given by that proposition have diameters bounded by
$$A_3 d^{-{1\over 2} \big(\ceil{(1-\zeta)n}-\floor{10\gamma_0n} \big)}  \leq d^{-{1-2\zeta\over 2} n}.$$
Here, we used our choices of $\zeta, \gamma_0$ and the fact that $n$ is large.  Therefore, these inverse branches do not contribute to the counting function $\sigma$.
In summary, Proposition \ref{p:inverse} gives us
$$\sum_s \sigma(s,s') \leq A_3 d^{-\floor{10\gamma_0 n}} d^{k\ceil{(1-\zeta)n}}.$$
It follows that
$$\sum_{s,s'} \sigma(s,s')  m_{s'}' \lesssim d^{-10\gamma_0 n} d^{k\ceil{(1-\zeta)n}} \sum_{s'} m_{s'}' \leq d^{-10\gamma_0 n} d^{k\ceil{(1-\zeta)n}},$$
where we have used that $\sum_{s'} m_{s'}' \leq 1$ since these numbers are the $\mu$ measures of disjoint cells inside $\P^k$.

Now, let $S\subset \{1,\ldots,N\}$ be the set of indexes $s$ such that $\fW_s$ is $j$-unsafe.  Assume by contradiction that the total $\mu$ measure of these cells is larger than $d^{-2\gamma_0 n}$,  
that is,  $\sum_{s\in S} m_s \geq d^{-2\gamma_0 n}$.  
From the definition of $r$, we see that $d^{-{1-2\zeta\over 2}n}<r^2$, so all inverse branches with $b\in (1-2r)\fW_s$ whose images are not contained in $(1-r)\fW_s$, contribute to the counting function $\sigma$.  Consider an $s\in S$ and an index $s'$ such that $\fW'_{s'}$ is not nice with $\fW_s$.  Recall that we are using $m=\ceil{(1-\zeta)n}$ which satisfies $d^{-m}<r^{30k}$ by the definition of $r$.  Then,  we can apply the first inequality in Corollary \ref{c:orbit-cell} for $m=\ceil{(1-\zeta)n}$, for $a$ the center of $\fW'_{s'}$ and for $\fW_s$ instead of $\W^{k,j}_{r,\eta}$. Using the notation in that corollary and in Definition \ref{d:good-cell}, we have 
$$ \sigma(s,s')\geq p-q_{r,\eta}\geq d^{-2\gamma_0 n}m_sd^{k\ceil{(1-\zeta)n}} ,$$
where we used that $n$ is big,  so $r$ is small and $A_6 r^{2k+2} d^{k\ceil{(1-\zeta)n}} \leq r^{2k+1} d^{k\ceil{(1-\zeta)n}}$.
Since $\fW_s$ is $j$-unsafe,  if $S'_s\subset \{1,\ldots,N'\}$ is the set of indexes $s'$ such that $\fW'_{s'}$ is not nice with $\fW_s$, then
$$\sum_{s'} \sigma(s,s')  m'_{s'} \geq d^{-2\gamma_0 n}m_sd^{k\ceil{(1-\zeta)n}} \sum_{s'\in S'_s} m_{s'}'  \geq d^{-4\gamma_0 n} m_s d^{k\ceil{(1-\zeta)n}} .$$ 
It follows that 
$$\sum_{s\in S,s'}   \sigma(s,s')  m'_{s'}  \geq \sum_{s\in S} d^{-4\gamma_0 n} m_s d^{k\ceil{(1-\zeta)n}} \geq d^{-6\gamma_0 n} d^{k\ceil{(1-\zeta)n}} $$ 
which contradicts the above estimates.  This proves that the total $\mu$ measure of $j$-unsafe cells is at most $d^{-2\gamma_0 n}$ and concludes the proof of the proposition.
\end{proof}

\begin{lemma} \label{l:safe}
There is a constant $A_8>0$ such that each safe cell $\W_{r,\eta}^k$ admits at least 
$$p_{r,\eta}:=\big[\mu\big((1-3r)\W_{r,\eta}^{k}\big) -d^{-2\gamma_0 n}\mu\big(\W_{r,\eta}^{k}\big) - r^{2k+1}\big](1-A_8d^{-\gamma_0n}) d^{kn}$$
inverse branches of order $n$ with images in $(1-r)\W_{r,\eta}^k$ and with diameters less than $d^{-{1-2\zeta\over 2}n}$.
\end{lemma}
\begin{proof}
We only need to consider $n$ large enough. Denote by $a$ the center of $\W_{r,\eta}^k$. Consider the set $R$ of points $b$ in $f^{-\floor{\zeta n}}(a)$ associated to inverse branches of order $\floor{\zeta n}$ on $\W_{r,\eta}^k$ with images in cells which are nice with $\W_{r,\eta}^k$. Observe that every point $b$ in $f^{-\floor{\zeta n}}(a)$ belongs to $R$ unless one of the following properties holds:
\begin{enumerate}
\item $b$ belongs to $3\bbS_r^j\cap {1\over 2}\Omega_j$ for some $j$;
\item $b$ belongs to $(1-2r)\W_{r,\eta}^{k,j}$ for some cell $\W_{r,\eta}^{k,j}$ which is not nice with $\W_{r,\eta}^k$;
\item $b$ belongs to $(1-2r)\W_{r,\eta}^{k,j}$ for some cell $\W_{r,\eta}^{k,j}$ which is nice with $\W_{r,\eta}^k$ but it is not associated to an inverse branch on 
$\W_{r,\eta}^{k,j}$ described in Proposition \ref{p:inverse} for $m=\floor{\zeta n}$ and  $\ell=\floor{10\gamma_0n}$.
\end{enumerate}
Note that the inverse branches from Proposition \ref{p:inverse} mentioned in (3) have diameters bounded by 
$$A_3 d^{-{1\over 2} \big(\floor{\zeta n}- \floor{10\gamma_0n}\big)} <r^2$$
because $n$ is large. Therefore, such a branch has image in $(1-r)\W_{r,\eta}^{k,j}$ which implies that $b\in R$. In (3), we ask $b$ not to be in this case.

Since $n$ is large, it is not difficult to check that $d^{-\floor{\zeta n}} <r^{30k}$.  Corollary \ref{c:orbit-street} implies that the number of $b$ satisfying (1) is bounded by $A_5 r d^{k\floor{\zeta n}}$.  We can also apply Corollary \ref{c:orbit-cell} for $m=\floor{\zeta n}$.  The second inequality in that corollary 
 implies that the number of $b$ satisfying (2) is bounded by 
$$\Big[\sum \Big( \mu(\W_{r,\eta}^{k,j})+A_6 r^{2k+2}\Big) \Big]d^{k\floor{\zeta n}},$$ 
where we only consider cells which are not nice with $\W_{r,\eta}^{k}$. 
Since $\W_{r,\eta}^k$ is safe, these cells have a total $\mu$ measure at most $d^{-\gamma_0 n}$.
Recall also that the total number of cells is $O(r^{-2k})$. We conclude that  the number of $b$ satisfying (2) is bounded by a constant times
$$\big[d^{-\gamma_0 n} + r^2\big] d^{k\floor{\zeta n}} \lesssim d^{-\gamma_0 n}  d^{k\floor{\zeta n}}.$$

By Proposition \ref{p:inverse}, the number of $b$ satisfying (3) is bounded by $A_3 d^{-10\gamma_0 n}  d^{k\floor{\zeta n}}$. We conclude that the set $R$ contains at least 
$(1-A_8d^{-\gamma_0 n}) d^{k\floor{\zeta n}}$ points for some constant $A_8>0$. As an inverse branch of order $n$ is the composition of an inverse branch of order $\floor{\zeta n}$ and an inverse branch of order $\ceil{(1-\zeta)n}$, the lemma follows from the definition of nice cells.
\end{proof}

\begin{lemma} \label{l:fixed}
Let $g:r\W^k\to (1-r)r\W^k$ be a holomorphic map whose image has diameter at most $d^{-{1-2\zeta\over 2}n}$. Then $g$ admits a unique fixed point $a$. Moreover, $a$ is attracting for $g$ and $\|Dg(a)\|\leq A_9 r^{-2} d^{-{1-2\zeta\over 2}n}$ for some constant $A_9>0$ independent of $r$ and $g$.  Futhermore,  for any point $a'\in r\W^k$, the sequence $g^n(a')$ converges to $a$.
\end{lemma}
\begin{proof}
The convex open set $r\W^k$ is Kobayashi hyperbolic and $g$ is strictly contracting for the Kobayashi metric on  $r\W^k$.  It follows that $g$ admits a unique fixed point $a$ which is moreover attracting. Since $a\in (1-r)r\W^k$, $g$ defines a holomorphic map from the ball of center $a$ and of radius $r^2$ to the ball of center $a$ and radius $d^{\frac{1-2\zeta}{2}n}$.

Let $\Fc$ be the family of holomorphic maps $h$  from the unit ball $\B_k$ to itself which fix the center. This family is normal and therefore $\|Dh(0)\|$ is uniformly bounded.  Applying this property to $g$ with  suitable scalings gives the desired estimate.
\end{proof}

\begin{proposition} \label{p:per}
For $n$ large enough, each safe cell $\W_{r,\eta}^k$ admits at least 
$p_{r,\eta}$ repelling periodic points $a$ of period $n$ which satisfy $\|Df^n(a)^{-1}\|\leq d^{-{1-\gamma\over 2}n}$ and belong to $J_k$. 
\end{proposition}
\begin{proof}
If $\W_{r,\eta}^k$ doesn't intersect $J_k$, then its $\mu$ measure is zero and $p_{r,\eta} \leq 0$,  so  the proposition is clear  in this case.    We assume then that $\W_{r,\eta}^k \cap J_k$ is non-empty.  We apply Lemma \ref{l:fixed} to each inverse branch $g: \W_{r,\eta}^k \to (1-r)\W_{r,\eta}^k$ in Lemma  \ref{l:safe} and get an attracting fixed point $a$ for $g$ which is a repelling periodic point of period $n$ for $f$. 
This point is obtained as the limit of $g^n(a')$ for any point $a'\in \W_{r,\eta}^k$. Choosing a point $a'$ in $J_k$ and using that 
$J_k$ is closed and invariant by $f^{-1}$,  we deduce that $a$ is on $J_k$. The estimate of the differential of $f^n$ also follows from Lemma  \ref{l:fixed} and the fact that 
$$A_9 r^{-2} d^{-{1-2\zeta\over 2}n} \leq d^{-{1-\gamma\over 2}n}$$
thanks to the definitions of $\zeta, r$ and the fact that $n$ is large.
\end{proof}

\section{Equidistribution of repelling periodic points} \label{s:equi}

In this section, we finish the proof of Theorem \ref{t:main}. By interpolation theory between Banach spaces \cite{triebel}, it is enough to consider the case where $\alpha=1$.  Consider a test function $\phi \in \Cc^1(\P^k)$.  Let $\chi:\R\to\R$ be such that $\chi = 0$ on $(-\infty,-1]$, $\chi(t)\geq t$ and $\chi(t)=t$ for $t\geq 1$.   Then,  we can write $\phi=\chi(\phi)- (\chi(\phi)-\phi)$ as a difference of two non-negative $\Cc^1$ functions.  Therefore, for simplicity, we can assume from now on that $\phi$ is non-negative. By multiplying $\phi$ by a constant we can also assume that $\|\phi\|_{\Cc^1}\leq 1$. Observe also that it is enough to consider $n$ big enough because otherwise the theorem is clear even when $P_{n,\gamma}$ is empty. 

\begin{lemma} \label{l:phi}
There exists a constant $A_{10}>0$ independent of $\phi$ such that for every $n\geq 0$
$$\Big\langle\frac{1}{d^{kn}}\sum_{a\in P_{n,\gamma}} \delta_a, \phi \Big\rangle \geq \langle \mu,\phi\rangle - A_{10} d^{-\gamma_0 n}.$$
\end{lemma}
\begin{proof}
We   use the notations of Sections \ref{s:branch} and \ref{s:periodic}.  In particular, we  use the cells $\W^k_{r,\eta}$ inside ${1\over 2}\Omega$ which cover ${1\over 10}\Omega$.
By using a partition of unity, we can assume that $\phi$ is supported by ${1\over 5}\Omega$. Denote by $\K \subset \Z^{2k}$ the set of indices $\eta$ such that $\W^k_{r,\eta}$ is safe.
Recall that the total number of cells is $O(r^{-2k})$. 
Denote by $\phi_\eta$ the infimum of $\phi$ on $\W^k_{r,\eta}$. 
Using Proposition \ref{p:per} and that $\|\phi\|_{\Cc^1} \leq1$, we have the following estimates (recall the definition of $p_{r,\eta}$ from Lemma \ref{l:safe})

\begin{eqnarray*}
\Big\langle \sum_{a\in P_{n,\gamma}} \delta_a, \phi \Big\rangle &\geq&  \sum_{\eta\in\K} p_{r,\eta} \phi_\eta \\
&=&  (1-A_8d^{-\gamma_0n}) d^{kn} \sum_{\eta\in\K} \big[\mu\big((1-3r)\W_{r,\eta}^{k}\big) -d^{-2\gamma_0 n}\mu\big(\W_{r,\eta}^{k}\big) - r^{2k+1}\big] \phi_\eta \\
& = & -(1-A_8d^{-\gamma_0n}) d^{kn} \sum_{\eta\in\K} \big[\mu\big(\W_{r,\eta}^{k})- \mu\big((1-3r)\W_{r,\eta}^{k}\big) \big) \big]\phi_\eta \\
&&  - (1-A_8d^{-\gamma_0n}) d^{kn} r^{2k+1}  \sum_{\eta\in\K} \phi_\eta \\
&&  -  (1-A_8d^{-\gamma_0n}) d^{kn} (1-d^{-2\gamma_0 n})\ \sum_{\eta\not\in\K} \mu\big(\W_{r,\eta}^{k}\big) \phi_\eta \\
&& - \big[1-(1-A_8d^{-\gamma_0n}) (1-d^{-2\gamma_0 n})\big] d^{kn}\ \sum_{\eta} \mu\big(\W_{r,\eta}^{k}\big) \phi_\eta \\
&& - d^{kn} \sum_{\eta} \big\langle \mu|_{\W_{r,\eta}^{k}}, \phi -\phi_\eta\big\rangle  + d^{kn}\sum_{\eta} \big\langle\mu|_{\W_{r,\eta}^{k}}, \phi\big\rangle .
\end{eqnarray*}

In the last sum, by Lemma \ref{l:Manhattan}, the first term is bounded from below by a positive constant times $-r d^{kn}$,   hence by a positive constant times $-d^{-\gamma_0 n} d^{kn}$. Since $\#\K=O(r^{-2k})$, the second term is bounded from below by a positive constant times $- rd^{kn} \gtrsim - d^{-\gamma_0 n}  d^{kn}$. The third term satisfies the same property thanks to Proposition \ref{p:unsafe}. The fourth term satisfies the same property because the factor before the sum satisfies it. The same holds for the fifth term because, since $\|\phi\|_{\Cc^1} \lesssim 1$,  we have 
$|\phi-\phi_\eta|\lesssim r \leq d^{-\gamma_0n}$ on each cell $\W_{r,\eta}^{k}$.
Finally, the last term is equal to $d^{kn} \langle \mu, \phi \rangle$ modulo some integral on the boundaries of the cells,  and the mass of $\mu$ on these boundaries can be bounded using Lemma \ref{l:Manhattan}. Hence,  this term is $d^{kn}\langle \mu, \phi \rangle+O(d^{-\gamma_0 n}) d^{kn}$.  This ends the proof.
\end{proof}

\begin{proof}[End of the proof of Theorem \ref{t:main}]
By applying  Lemma \ref{l:phi} to $\phi=1$, we get $\#P_{n,\gamma}\geq d^{kn} -A_{10} d^{-\gamma_0n} d^{kn}$. 
Recall also that $\# P_n = d^{kn}+ O(d^{-\gamma_0n} d^{kn})$, see e.g. \cite[Proposition 1.3]{dinh-sibony:cime}. Here, the points in $P_n$ are counted with multiplicities.
We deduce that $\#(P_n\setminus P_{n,\gamma}) = O(d^{-\gamma_0n} d^{kn})$. Note that $\#P_{n,\gamma}\leq \#P_n$ because $P_{n,\gamma} \subset P_n$.

By Lemma  \ref{l:phi},  it remains to show that 
$$\Big\langle\frac{1}{d^{kn}}\sum_{a\in P_n} \delta_a, \phi \Big\rangle \leq \langle \mu,\phi\rangle + A d^{-\gamma_0 n}$$
for some constant $A>0$.  Applying Lemma \ref{l:phi}   to $1-\phi$ instead of $\phi$  gives
$${\# P_{n,\gamma} \over d^{kn}} - \Big\langle\frac{1}{d^{kn}}\sum_{a\in P_{n,\gamma}} \delta_a, \phi \Big\rangle 
=\Big\langle\frac{1}{d^{kn}}\sum_{a\in P_{n,\gamma}} \delta_a, 1-\phi \Big\rangle \geq 1- \langle \mu,\phi\rangle - A_{10} d^{-\gamma_0 n}$$
which implies
$$\Big\langle\frac{1}{d^{kn}}\sum_{a\in P_{n,\gamma}} \delta_a, \phi \Big\rangle \leq  \langle \mu,\phi\rangle + O(d^{-\gamma_0 n}).$$
Finally,  we have
\begin{eqnarray*}
\Big\langle\frac{1}{d^{kn}}\sum_{a\in P_n} \delta_a, \phi \Big\rangle &=& \Big\langle\frac{1}{d^{kn}}\sum_{a\in P_{n,\gamma}} \delta_a, \phi \Big\rangle
+ \Big\langle\frac{1}{d^{kn}}\sum_{a\in P_n\setminus P_{n,\gamma}} \delta_a, \phi \Big\rangle.
\end{eqnarray*}
We have seen that the first term in the last sum is bounded by $\langle \mu,\phi\rangle + O(d^{-\gamma_0 n})$. The second term is equal to $O(d^{-\gamma_0 n})$ thanks to the above discussion on the cardinalities of $P_n$ and $P_{n,\gamma}$. This completes the proof.
\end{proof}

\begin{proof}[Proof of Corollary \ref{c:main}]
The two sets of periodic points mentioned in this corollary are subsets of $P_n\setminus P_{n,\gamma}$. Therefore, the corollary follows from the above discussion about the cardinality of the last set.
\end{proof}

\begin{conjecture} \label{cj:optimal}
Let $f$ be a holomorphic endomorphism of $\P^k$ of algebraic degree $d\geq 2$ and $\mu$ be its equilibrium measure. 
Then,  for every constant $1<\lambda< d^{1/2} $ the following property holds for some constants $0<\gamma<1$ and $A>0$. Let $P_n$ be the set of periodic points of period $n$ of $f$. Let $P_{n,\gamma}$ be the set of points $a\in P_n\cap J_k$ such that $\|Df^n(a)^{-1}\|\leq d^{-{1-\gamma\over 2}n}$. Then, we have
$$\Big| \Big\langle\frac{1}{d^{kn}}\sum_{a\in P_{n,\gamma}} \delta_a -\mu, \phi \Big\rangle \Big| \leq A \lambda^{-n} \|\phi\|_{\Cc^1},$$
for any $\Cc^1$ test function $\phi$ on $\P^k$.
\end{conjecture}

\begin{conjecture} \label{cj:opttimal-point}
Let  $A_n$  denote the number of non-repelling periodic points of order $n$ of $f$ and $B_n$ be the number of periodic points of order $n$ outside the small Julia set $J_k$, counting  multiplicities.   Then,  for every constant  $1<\lambda< d^{1/2} $  we have  $A_n= O(  \lambda^{-n}  d^{kn})$ and $B_n= O(  \lambda^{-n}  d^{kn})$ as $n$ tends to infinity.
\end{conjecture}



\begin{thebibliography}{FLMn83}
\bibitem[BD99]{briend-duval:acta}
Jean-Yves Briend and Julien Duval.
\newblock Lyapunov exponent and the distribution of the periodic points of an
  endomorphism of {{\(\mathbb {CP}^k\)}}.
\newblock {\em Acta Math.}, 182(2):143--157, 1999.

\bibitem[BD01]{briend-duval:IHES}
Jean-Yves Briend and Julien Duval.
\newblock Two characterizations of equilibrium measure of an endomorphism of
  {{\(P^k(\mathbb{C})\)}}.
\newblock {\em Publ. Math., Inst. Hautes {\'E}tud. Sci.}, 93:145--159, 2001.

\bibitem[Bro65]{brolin}
Hans Brolin.
newblock Invariant sets under iteration of rational functions.
\newblock {\em Ark. Mat.}, 6:103--144 (1965), 1965.

\bibitem[DNS10]{DNS:JDG}
Tien-Cuong Dinh, Vi\^{e}t-Anh Nguy\^{e}n, and Nessim Sibony.
\newblock Exponential estimates for plurisubharmonic functions and stochastic
  dynamics.
\newblock {\em J. Differential Geom.}, 84(3):465--488, 2010.

\bibitem[DS10a]{dinh-sibony:cime}
Tien-Cuong Dinh and Nessim Sibony.
\newblock Dynamics in several complex variables: endomorphisms of projective
  spaces and polynomial-like mappings.
\newblock In {\em Holomorphic dynamical systems}, volume 1998 of {\em Lecture
  Notes in Math.}, pages 165--294. Springer, Berlin, 2010.

\bibitem[DS10b]{dinh-sibony:equid-speed}
Tien-Cuong Dinh and Nessim Sibony.
\newblock Equidistribution speed for endomorphisms of projective spaces.
\newblock {\em Math. Ann.}, 347(3):613--626, 2010.

\bibitem[DY25]{DY25}
Tien-Cuong Dinh and Jit~Wu Yap.
\newblock {\em Unpublished note}, 2025.


\bibitem[FRL06]{favre-rivera-letelier}
Charles Favre and Juan Rivera-Letelier.
\newblock \'{E}quidistribution quantitative des points de petite hauteur sur la
  droite projective.
\newblock {\em Math. Ann.}, 335(2):311--361, 2006.

\bibitem[FS94]{fornaess-sibony:93}
John~Erik Forn{\ae}ss and Nessim Sibony.
\newblock Complex dynamics in higher dimensions.
\newblock In {\em Complex potential theory. Proceedings of the NATO Advanced
  Study Institute and S\'eminaire de math\'ematiques superi\'eures, Montr\'eal,
  Canada, July 26 - August 6, 1993}, pages 131--186. Dordrecht: Kluwer Academic
  Publishers, 1994.

\bibitem[FS01]{fornaess-sibony:examples}
John~Erik Forn{\ae}ss and Nessim Sibony.
\newblock Dynamics of {{\(\mathbb{P}^2\)}} (examples).
\newblock In {\em Laminations and foliations in dynamics, geometry and
  topology. Proceedings of the conference held at SUNY at Stony Brook, USA, May
  18--24, 1998}, pages 47--85. Providence, RI: American Mathematical Society
  (AMS), 2001.
  
\bibitem[FLM83]{freire-lopes-mane}
Alexandre Freire, Artur Lopes, and Ricardo Ma\~n\'e.
\newblock An invariant measure for rational maps.
\newblock {\em Bol. Soc. Brasil. Mat.}, 14(1):45--62, 1983.  

\bibitem[GV25]{gauthier-vigny:periodic}
Thomas Gauthier and Gabriel Vigny.
\newblock Quantitative equidistribution of periodic points for rational maps.
\newblock {\em Preprint,upcoming}, 2025.

\bibitem[HP94]{hubbard-papadopol}
John~H. Hubbard and Peter Papadopol.
\newblock Superattractive fixed points in {{\({\mathbb{C}}^ n\)}}.
\newblock {\em Indiana Univ. Math. J.}, 43(1):321--365, 1994.

\bibitem[Kau17]{kaufmann:skoda}
Lucas Kaufmann.
\newblock A {S}koda-type integrability theorem for singular {M}onge-{A}mp\`ere
  measures.
\newblock {\em Michigan Math. J.}, 66(3):581--594, 2017.

\bibitem[Lel57]{lelong}
Pierre Lelong.
\newblock Int{\'e}gration sur un ensemble analytique complexe.
\newblock {\em Bull. Soc. Math. Fr.}, 85:239--262, 1957.

\bibitem[Lyu82]{lyubich2}
Mikhail Lyubich.
\newblock The measure of maximal entropy of a rational endomorphism of a
  {R}iemann sphere.
\newblock {\em Funktsional. Anal. i Prilozhen.}, 16(4):78--79, 1982.

\bibitem[Lyu83]{lyubich}
Mikhail Lyubich.
\newblock Entropy properties of rational endomorphisms of the {R}iemann sphere.
\newblock {\em Ergodic Theory Dynam. Systems}, 3(3):351--385, 1983.

\bibitem[Oku16]{okuyama}
Y\^{u}suke Okuyama.
\newblock Effective divisors on the projective line having small diagonals and
  small heights and their application to adelic dynamics.
\newblock {\em Pacific J. Math.}, 280(1):141--175, 2016.

\bibitem[SW80]{sibony-wong}
Nessim Sibony and Pit-Mann Wong.
\newblock Some results on global analytic sets.
\newblock Semin. {P}. {Lelong} - {H}. {Skoda}, {Analyse}, {Annees} 1978/79,
  {Lect}. {Notes} {Math}. 822, 221-237 (1980)., 1980.

\bibitem[Tri78]{triebel}
Hans Triebel.
\newblock {\em Interpolation theory, function spaces, differential operators},
  volume~18 of {\em North-Holland Mathematical Library}.
\newblock North-Holland Publishing Co., Amsterdam-New York, 1978.

\bibitem[Yap24]{yap}
Jit~Wu Yap.
\newblock Quantitative equidistribution of small points for canonical heights.
\newblock {\em arxiv:2410.21679}, 2024.

\end{thebibliography}
\end{document}